\documentclass{birkjour}
\usepackage[utf8]{inputenc}
\usepackage{amsmath,amsfonts,amssymb,amscd,amsthm,xspace,mathtools}
\usepackage{tikz}
\usepackage{float}
\theoremstyle{plain}
\newtheorem{theorem}{Theorem}[section]
\newtheorem{example}{Example}
\newtheorem{remark}{Remark}
\newtheorem{definition}{Definition}
\newtheorem{lemma}[theorem]{Lemma}

 \numberwithin{equation}{section}
\newcommand{\n}[1]{\underline{#1}}
\newcommand{\EF}[1]{\operatorname{\mathcal{Z}}(#1)}
\newcommand{\F}[1]{\operatorname{\mathcal{L}}(#1)}
\newcommand{\minor}[3]{#1|#2/#3}
\newcommand{\ess}[1]{\operatorname{e}_{#1}}
\newcommand{\cl}[1]{\operatorname{cl}_{#1}}
\newcommand{\rank}[1]{\operatorname{r}_{#1}}
\newcommand{\Unif}[2]{\operatorname{U}_{#1,#2}}
\newcommand{\Z}{\mathbb Z}

\DeclareMathOperator{\x}{x}
\DeclareMathOperator{\y}{y}
\DeclareMathOperator{\Pot}{Pot}
\DeclareMathOperator{\Zxy}{\Z[\x,\y]}
\DeclareMathOperator{\Zx}{\Z[\x]}
\DeclareMathOperator{\Zy}{\Z[\y]}

\newcommand{\Aut}{\operatorname{Aut}}
\newcommand{\EM}{\mathcal E(M)}

\newcommand{\rgpMxy}[3]{\operatorname{S}(#1;#2, #3)}
\newcommand{\tuttexy}[3]{\operatorname{T}(#1;#2, #3)}
\newcommand{\rgpM}[1]{\rgpMxy{#1}{\x}{\y}}
\newcommand{\tutte}[1]{\tuttexy{#1}{\x}{\y}}
\newcommand{\cloud}[2]{\operatorname{c}(#1, #2; \x)}
\newcommand{\flock}[2]{\operatorname{f}(#1, #2; \y)}
\newcommand{\cloudC}[3]{\operatorname{\underline{c}}(#1, #2, #3; \x)}
\newcommand{\flockC}[3]{\operatorname{\underline{f}}(#1, #2, #3; \y)}
\newcommand{\bx}[2]{\operatorname{c}_{#2,#1}(\x)}
\newcommand{\by}[2]{\operatorname{f}_{#2,#1}(\y)}
\DeclareMathOperator{\dx}{d_{\x}}
\DeclareMathOperator{\dy}{d_{\y}}
\newcommand{\we}{\operatorname{A}}

\title[Cyclic Flats and the \textsc{Tutte} Polynomial]{Computing the \textsc{Tutte} Polynomial of a Matroid from its Lattice of Cyclic Flats}

\author[Jens Niklas Eberhardt]{Jens Niklas Eberhardt}

\address{%
RWTH Aaachen\\
 Templergraben 64\\
52062 Aachen\\
Germany}

\email{jens.eberhardt@rwth-aachen.de}

\subjclass{Primary 05B35; Secondary 00A00}

\keywords{Matroid theory, Tutte polynomial, rank generating polynomial, cyclic flats}

\date{July 7, 2013}

\begin{document}

\maketitle

\begin{abstract}
We show how the \textsc{Tutte} polynomial of a matroid $M$ can be computed from its \emph{condensed configuration}, which is a statistic of its lattice of cyclic flats. The results imply that the \textsc{Tutte} polynomial of $M$ is already determined by the abstract lattice of its cyclic flats together with their cardinalities and ranks. They furthermore generalize a similiar statement for perfect matroid designs due to Mphako \cite{PMD} and help to understand families of matroids with identical \textsc{Tutte} polynomial as constructed in \cite{LargeFamTutte}.
\end{abstract}
\section{Introduction}
The \textsc{Tutte} polynomial is a central invariant in matroid theory. But passing over from a matroid $M$ to its \textsc{Tutte} polynomial $\tutte{M}$ generally means a big loss of information. This paper gives one explanation for this phenomenon by showing how little information about the \textit{cyclic flats} of a matroid is really needed for the computation of its \textsc{Tutte} polynomial.

From now on let $M$ be a matroid. A flat $X$ in $M$ is called \emph{cyclic} if $M|X$ contains no coloops. Section \ref{secBackground} will recapitulate some basic facts about cyclic flats and show how the \textsc{Tutte} polynomial can be expressed in terms of \emph{cloud} and \emph{flock polynomials} of cyclic flats as introduced by Plesken in \cite{Plesken}. Then Section \ref{secIdentities} establishes some important identities for cloud and flock polynomials needed later on.

$\EF{M}$, the set of cyclic flats of $M$, is a lattice w.r.t. inclusion (c.f. Figure \ref{figConf}). In Section \ref{secConfiguration} we introduce the \emph{configuration} of $M$: the abstract lattice of its cyclic flats\footnote{By this we formally mean the isomorphism class of the lattice $(\EF{M}, \subseteq).$} together with their cardinalities and ranks. We then prove:
\newtheorem*{thmConfTutte}{Theorem \textbf{\ref{thmConfTutte}}}
\begin{thmConfTutte}
 The \textsc{Tutte} polynomial of a matroid is determined by its configuration.
\end{thmConfTutte}
\begin{figure}[t]
\begin{minipage}{0.3\textwidth}
\begin{center}
\begin{tikzpicture}
  \coordinate[label=above left:$b$] (A) at (1,1.5);
  \coordinate[label=above left:$c$] (B) at (2,2);
  \coordinate[label=above left:$a$](C) at (0,1);
  \coordinate[label=below left:$d$] (D) at (1,0.5);
  \coordinate[label=below:$e$]  (E) at (2,0);
  \coordinate[label=below right:$f$] (F) at (2.1,1.1);
  \draw (A) -- (B) -- (C) 
  (C) --  (D) --  (E) ;
  \fill (A) circle (3pt) (B) circle (3pt) (C) circle (3pt) (D) circle (3pt) (E) circle (3pt) (F) circle (3pt) ;
\end{tikzpicture} \\
$M_1$\\\vspace{7pt}
\begin{tikzpicture}
  \coordinate[label=left:$a\dots f$] (B) at (1,3);
  \coordinate[label=left:$abc$](A) at (0,2);
  \coordinate[label=left:$\emptyset$] (D) at (1,1);
  \coordinate[label=left:$ade$]  (E) at (2,2);
  \draw (A) -- (B) --  (E)  -- (D) -- (A);
  \fill (A) circle (3pt) (B) circle (3pt) (D) circle (3pt) (E) circle (3pt)  ;
\end{tikzpicture}
\\
$\EF{M_1}$
\end{center}
\end{minipage}
\begin{minipage}{0.3\textwidth}
\begin{center}
\begin{tikzpicture}
  \coordinate[label=above left:$b$] (B) at (1,1.5);
  \coordinate[label=above left:$c$] (C) at (2,2);
  \coordinate[label=above left:$a$](A) at (0,1);
  \coordinate[label=below right:$d$] (D) at (1,0.5);
  \coordinate[label=below right:$e$]  (E) at (2,1);
  \coordinate[label=below right:$f$] (F) at (3,1.5);
  \draw (A) -- (B) -- (C) 
  (D) --  (E) --  (F) ;
  \fill (A) circle (3pt) (B) circle (3pt) (C) circle (3pt) (D) circle (3pt) (E) circle (3pt) (F) circle (3pt) ;
\end{tikzpicture}\\\vspace{10pt}
$M_2$\\\vspace{7pt}
\begin{tikzpicture}
  \coordinate[label=left:$a\dots f$] (B) at (1,3);
  \coordinate[label=left:$abc$](A) at (0,2);
  \coordinate[label=left:$\emptyset$] (D) at (1,1);
  \coordinate[label=left:$def$]  (E) at (2,2);
  \draw (A) -- (B) --  (E)  -- (D) -- (A);
  \fill (A) circle (3pt) (B) circle (3pt) (D) circle (3pt) (E) circle (3pt)  ;
\end{tikzpicture}\\
$\EF{M_2}$
\end{center}
\end{minipage}
\begin{minipage}{0.3\textwidth}
\vspace{0pt}\hspace{10pt}
\begin{tikzpicture}
  \coordinate[label=above:{$(6,\, 3)$}] (B) at (1,3);
  \coordinate[label=left:{$(3,\, 2)$}](A) at (0,2);
  \coordinate[label=below:{$(0,\, 0)$}] (D) at (1,1);
  \coordinate[label=right:{$(3,\, 2)$}]  (E) at (2,2);
  \draw (A) -- (B) --  (E)  -- (D) -- (A);
  \fill (A) circle (3pt) (B) circle (3pt) (D) circle (3pt) (E) circle (3pt)  ;
\end{tikzpicture} 
\end{minipage}
\label{figConf}
\caption{Left: Two non isomorphic matroids and their lattices of cyclic flats. Right: Their configuration (the labels are $(|X|, \rank{M_i}(X))$ for $X\in \EF{M_i}$).}
\end{figure}
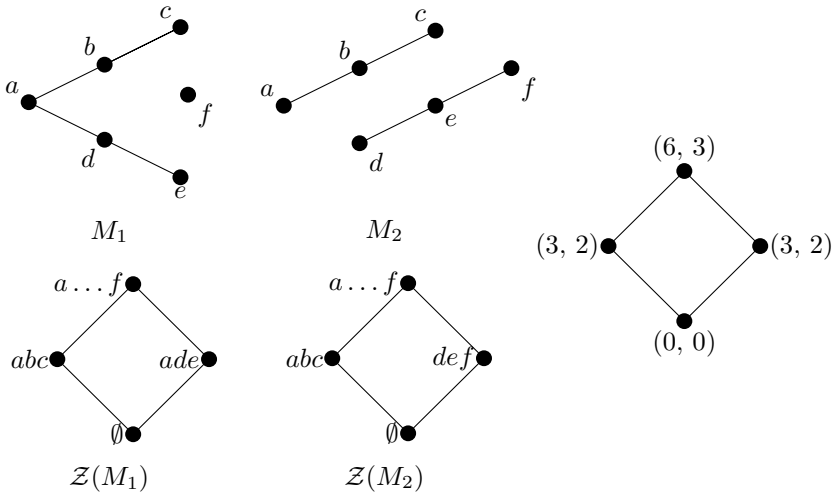
While $M$ is determined by its cyclic flats and their ranks (c.f. \cite{Bonin}), it generally is far from being determined by its configuration (c.f. Figure \ref{figConf}); there are even superexponential families of matroids with identical configurations (c.f. \cite{LargeFamTutte}).
So Theorem \ref{thmConfTutte} explains one big part of the information lost when passing from $M$ to its \textsc{Tutte} polynomial.

In Section \ref{secCondensedConfiguration} we incorporate symmetries in $M$ to shrink down the information needed for its  \textsc{Tutte} polynomial even more. Let $G\leq\Aut(M)$, $P$ be the set of $G$-orbits of $\EF{M}$ and $\{R_B\}_{B\in P}$ a system of representatives.
The \emph{condensed configuration} of $M$ corresponding to $P$ consists of the cardinalities and ranks of the $R_B$ and the matrix\footnote{This \emph{generalized adjacency matrix} was introduced by Plesken in \cite{Plesken2}; for $G=\{1\}$ it is simply the adjacency matrix of the lattice $\EF{M}$.} $(\we_P(B,C))_{B,C\in P}$ where
$$\we_P(B,C):=|\{X\in B| B\subseteq R_C\}|.$$

After discussing some examples, e.g.  a condensed configuration for the \textsc{Golay} code matroid, we will prove:
\newtheorem*{thmCondConfTutte}{Theorem \textbf{\ref{thmCondConfTutte}}}
\begin{thmCondConfTutte}
 The \textsc{Tutte} polynomial of $M$ is determined by a condensed configuration of $M$.
\end{thmCondConfTutte}
Section \ref{secPMD} shows how to obtain a condensed configuration of a perfect matroid design using only the cardinalities of flats of given rank. Together with Theorem \ref{thmCondConfTutte} this yields a new proof for Mphako's results about the \textsc{Tutte} polynomial of perfect matroid designs in \cite{PMD}.
\subsection*{Acknowledgments}
I would like to thank Wilhelm Plesken for many helpful discussions and Joseph E. Bonin for providing useful literature.
\section{Background} \label{secBackground}
We quickly recapitulate the most important facts about cyclic flats and the cloud/flock formula for the rank generating polynomial from \cite{Plesken}, while assuming familiarity with the basics of matroid theory.
From now on let $M$ be a matroid without loops and coloops with rank funktion $\rank{M}$, closure operator $\cl{M}$ and ground set $\EM$.
\subsection{Cyclic flats}
A flat $X$ in $M$ is called a \emph{cyclic flat} if $M|X$, the restriction of $M$ to $X$, contains no coloops. We denote the set of (cyclic) flats by $\F{M}$, resp. $\EF{M}$; both form a lattice w.r.t. inclusion.
\begin{example} \label{expUniformCF} Let $r<n$ and consider $\Unif{r}{n}$, the uniform matroid of rank $r$ on $n$ points. Then $\F{\Unif{r}{n}}=\Pot_{<r}(\n{n}) \cup \{\n{n}\}$ and $\EF{\Unif{r}{n}}=\{\emptyset ,\n{n}\}$. If vice versa $|\EF{M}|=2$, then $M$ is a uniform matroid.
\end{example}
An arbitrary flat in $M$ contains a unique maximal cyclic flat $\ess{M}(X)$, obtained by removing the coloops in $M|X$ from $X$. This induces a surjective lattice homomorphism$$\ess{M}: \F{M} \rightarrow\EF{M}.$$

Like the lattice of flats, the lattice of cyclic flats behaves well w.r.t. restriction and contraction:

\begin{remark}\label{remMinor} Let $X\subseteq Y$ be cyclic flats in $M$, then $$[X,Y]_{\EF{M}} \rightarrow \EF{M|Y/X},\, Z\mapsto Z-X$$ is an isomorphism of lattices and furthermore $$\rank{M|Y/X}(Z-X)=\rank{M}(Z)-\rank{M}(X).$$
\end{remark}
\subsection{Cloud/flock formula for the rank generating polynomial}
Instead of the \textsc{Tutte} polynomial $\tutte{M}$ of $M$ we will from now on study its \emph{rank generating polynomial} $\rgpM{M}$ which is $$\rgpM{M}=\sum_{X\subseteq\EM} \x^{\rank{M}(E)-\rank{M}(X)}\y^{|X|-\rank{M}(X)}.$$
\textsc{Tutte} and rank generating polynomials can be transformed into each other since $\tutte{M}=\rgpMxy{M}{\x-1}{\y-1}$.

The summands of the rank generating polynomial of $M$ can be conveniently resorted using the cloud and flock polynomials of its cyclic flats introduced in \cite{Plesken}\footnote{We note that the cloud/flock formula for the rank generating polynomial looks very similiar to a formula for the \textsc{Tutte} polynomial in \cite{TutteIdent} but is in fact not equivalent.}:
\begin{definition}
Let $X$ be a cyclic flat in $M$. The \emph{cloud}, resp. \emph{flock} \emph{polynomial} of $X$ in $M$ is defined by \begin{align*}  \cloud{M}{X}&:= \sum_{\mathclap{Y\in \ess{M}^{-1}(\{X\})}}\x^{\rank{M}(M)-\rank{M}(Y)}\text{ , resp.}\\   \flock{M}{X}&:=\sum_{\mathclap{Y \in \cl{M}^{-1}(\{X\})}} \y^{|Y|-\rank{M}(Y)}
\end{align*}
\end{definition}
\begin{theorem}[Cloud/flock formula\cite{Plesken}]
$$\rgpM{M}=\sum_{X\in \EF{M}} \cloud{M}{X} \flock{M}{X}.$$
\end{theorem}
\begin{proof}[Proof (Sketch):]Notice that $\EM=\biguplus_{X\in\EF{M}}\cl{M}^{-1}(\ess{M}^{-1}(\{X\}))$ and that corank is constant on the fibers of $\cl{M}$ and nullity (cardinality minus rank) is constant on the fibres of $\ess{M}$.\end{proof}
\begin{example}
We again consider the uniform matroid $\Unif{r}{n}$ for $r<n$. Then $\cl{\Unif{r}{n}}^{-1}(\{\emptyset\})=\{\emptyset\}$ and $\ess{\Unif{r}{n}}^{-1}(\{\n{n}\})=\{\n{n}\}$, hence
$$\flock{\Unif{r}{n}}{\emptyset}=\cloud{\Unif{r}{n}}{\n{n}}=1.$$
Furthermore $\ess{\Unif{r}{n}}^{-1}(\{\emptyset\})=\Pot_{<r}(\n{n})$, hence
$$\cloud{\Unif{r}{n}}{\emptyset}=\sum_{0\leq i < r}  \binom {n} {i} \x^{r-i}.$$
Analogously $\cl{\Unif{r}{n}}^{-1}(\{\n{n}\})=\Pot_{\geq r}(\n{n})$ and
$$\flock{\Unif{r}{n}}{\n{n}}=\sum_{r\leq i\leq n}\binom {n} {i} \y^{i-r}.$$
\end{example}

\section{Identities for cloud and flock polynomials} \label{secIdentities}
We state some useful identities for cloud and flock polynomials and show that the cloud and flock polynomial of the empty set and the ground set are already determined by the cloud and flock polynomials of all other cyclic flats. This is crucial for the recursive algorithms introduced later on. 

The rank generating polynomial $\rgpM{M}$ of $M$ is per definitionem a sum over all subsets of $\EM$ and it is easy to see that for $n=|\EM|$ and $r=\rank{M}(M)$
\begin{align}\label{eqRGPbinom}
\rgpMxy{M}{\x}{\x^{-1}}=\sum_{i=0}^{n} \binom{n}{i} \x^{r-i}.
\end{align}
To make use of this identity we define the $\Z$-linear maps\footnote{Notice the subtle difference between the definitions of $\dx$ and $\dy$: $\dx$ cuts of the constant term, whereas $\dy$ does not.}
\begin{align*}
	\dx &: \Zxy \rightarrow \Zx,\, f \mapsto \sum_{1\leq i} a_i \x^i,\text{ for }  f(\x,\x^{-1})=\sum_{i} a_i \x^i\text{ and}\\
       \dy &: \Zxy \rightarrow \Zy,\,  f \mapsto \sum_{0\leq i} b_i \y^i,\text{ for }f(\y^{-1},\y)=\sum_{i} b_i \y^i.
\end{align*}
And furthermore a notation for the cloud and flock polynomials in the uniform matroid $\Unif{r}{n}$. For $r<n$:
 \begin{align*}
	\bx{n}{r}&:=\cloud{\Unif{r}{n}}{\varnothing}=\sum_{0\leq i < r}  \binom {n} {i} \x^{r-i} \text{ and} \label{eqCloudbinom}\\
	\by{n}{r}&:=\flock{\Unif{r}{n}}{\n{n}}=\sum_{r\leq i\leq n}\binom {n} {i} \y^{i-r}.
 \end{align*}
 and for $n=r=0:$
  \begin{align*}
	\bx{0}{0}&:=\by{0}{0}=1.
 \end{align*}

In this notation equation \eqref{eqRGPbinom} becomes
\begin{align}
\dx(\rgpM{M})&=\bx{n}{r} \text{ and}\\
\dy(\rgpM{M})&=\by{n}{r}.
\end{align}
Applying these identities to the cloud/flock formula of the rank generating polynomial we obtain
\begin{lemma}\label{lemmaCFESGS}
The cloud (flock) polynomial of the empty (ground) set only depends on the cloud and flock polynomials of the cyclic flats $X\not\in\{\emptyset, \EM\}$:
\begin{align*}
   \cloud{M}{\varnothing}&= \bx{n}{r} - \dx\bigl[\sum_{X}\cloud{M}{X}\flock{M}{X}\bigr] \text{ and}\\
   \flock{M}{E}&=\by{n}{r} - \dy\bigl[\sum_{X} \cloud{M}{X}\flock{M}{X}\bigr],
 \end{align*}
where $X\in \EF{M}-\{\emptyset, \EM\}$, $n$ is the number of points and $r$ the rank of $M$.
\end{lemma}
\begin{proof}
 Equation \eqref{eqCloudbinom} yields
 $$\bx{n}{r}=\dx(\rgpM{M})=\dx\bigl[\,\,\sum_{\mathclap{X\in \EF{M}}}\,\,\cloud{M}{X}\flock{M}{X}\bigr].$$
Using the $\Z$-linearity of $\dx$ we solve for $\cloud{M}{\varnothing}$
$$
\dx(\cloud{M}{\varnothing}\flock{M}{\varnothing})= \bx{n}{r} -\dx\bigl[\,\,\sum_{\mathclap{\substack{X\in \EF{M}\\X\neq \emptyset}}}\,\,\cloud{M}{\emptyset}\flock{M}{X}\bigr].$$
Now firstly $\flock{M}{\varnothing}=\cloud{M}{\EM}=1$. Secondly $\dx(\cloud{M}{\varnothing})=\cloud{M}{\varnothing}$, since $M$ does not consist of coloops and hence $\cloud{M}{\varnothing}$ has no constant term. Furthermore $\dx(\flock{M}{\EM})=0$, since $\dx$ cuts of the constant term. This yields the first statement. The second one follows analogously.
\end{proof}

Another crucial fact is that the cloud (flock) polynomial of a cyclic flat $X$ only depends on $M/X$ ($M|X$) as the following lemma states.
\begin{lemma}\label{lemmaCFminor} Let $X$ be a cyclic flat in $M$. Then
\begin{align*}
  \cloud{M}{X}&=\cloud{M/X}{\emptyset}\text{ and}\\
  \flock{M}{X}&=\flock{M|X}{X}.
 \end{align*}
\end{lemma}
\section{Configurations}\label{secConfiguration}
We introduce the \emph{configuration} of a matroid and show how to compute the rank generating polynomial only using the configuration.
\begin{definition}
The \emph{configuration} of $M$ consists of
\begin{enumerate}
\item the abstract lattice of cyclic flats  in $M$, i.e. the isomorphism class of the lattice $(\EF{M}, \subseteq).$, together with
\item the cardinalities and ranks of the cyclic flats.
\end{enumerate} 
\end{definition}
The configuration of a matroid describes \emph{how it is is build up from uniform matroids} as the following example illustrates:
\begin{example} \label{expConf}
1.) According to Example \ref{expUniformCF}, a matroid without coloops is uniform on $n$ points and of rank $r$ ($r<n$) iff its configuration is
\begin{center}
\begin{tikzpicture}
  \coordinate[label=left:{$(n,\, r)$}] (B) at (0,.5);
  \coordinate[label=left:{$(0,\, 0)$}] (D) at (0,0);
  \draw (B) -- (D) ;
  \fill (B) circle (3pt) (D) circle (3pt) ;
\end{tikzpicture} 
\end{center}
2.) According to Remark \ref{remMinor}, the configuration of the minor $\minor{M}{Y}{X}$ (for cyclic flats $X\subset Y$) corresponds to the interval $[X,Y]$ in the configuration of $M$. 

Consider for example the matroid $M_i$ ($i=1,2$) from Figure \ref{figConf} and let $X$ be one of the cyclic flats of rank 2. Then from the configuration of $M_i$ one can deduce that
$M|X\cong U_{2,3}$ and $M/X\cong U_{1,3}$ since their configurations are
 \begin{center}
   \begin{tikzpicture}
  \coordinate[label=left:{$(3,\, 2)$}] (B) at (0,.5);
  \coordinate[label=left:{$(0,\, 0)$}] (D) at (0,0);
  \draw (B) -- (D) ;
  \fill (B) circle (3pt) (D) circle (3pt) ;
\end{tikzpicture}  \hspace{30pt}\vspace{-20pt}
  \begin{tikzpicture}
  \coordinate[label=left:{$(3,\, 1)$}] (B) at (0,.5);
  \coordinate[label=left:{$(0,\, 0)$}] (D) at (0,0);
  \draw (B) -- (D) ;
  \fill (B) circle (3pt) (D) circle (3pt) ;
\end{tikzpicture}
  \vspace{-2pt}
 \end{center}
 \begin{center}
  \hspace{40 pt} and \hspace{40 pt}.
 \end{center}
\end{example}
The main motivation for introducing configurations is the following main theorem.
\begin{theorem}\label{thmConfTutte}
The rank generating polynomial of a matroid can be computed by its configuration.
\end{theorem}
\begin{proof}
We show how to recursively compute the cloud and flock polynomials for all cyclic flats $X$ in $M$.\\
\emph{Case 1:} $X\not\in\{ \emptyset,\EM\}$. Using Lemma \ref{lemmaCFminor} we have
\begin{align*}
  \cloud{M}{X}&=\cloud{M/X}{\emptyset}\text{ and}\\
  \flock{M}{X}&=\flock{M|X}{X}.
 \end{align*}
Now the configurations of $M/X$ and $M|X$ correspond to proper intervals in the configuration of $M$ (c.f. Remark \ref{remMinor} and Example \ref{expConf}), hence these cloud and flock polynomials can be computed recursively.\\
\emph{Case 2:}  $X \in\{ \emptyset,\EM\}$. Firstly
$$\cloud{M}{\EM}=\flock{M}{\emptyset}=1.$$
 Then by Lemma \ref{lemmaCFESGS} we have
\begin{align*}
   \cloud{M}{\varnothing}&= \bx{n}{r} - \dx\bigl[\sum_{X}\cloud{M}{X}\flock{M}{X}\bigr] \text{ and}\\
   \flock{M}{E}&=\by{n}{r} - \dy\bigl[\sum_{X} \cloud{M}{X}\flock{M}{X}\bigr],
 \end{align*}
where $X\in \EF{M}-\{\emptyset, \EM\},\,n=|\EM|$ and $r=\rank{M}(\EM)$. So Case 2 reduces to Case 1 and the statement follows using the cloud/flock formula for the rank generating polynomial.
\end{proof}
Theorem \ref{thmConfTutte} yields a better understanding in matroids with identical rank generating polynomials. In \cite{LargeFamTutte} superexponential familes of matroids with identical rank generating polynomials are constructed; as it turns out they all have - by construction - the same configuration.
\section{Condensed configurations}\label{secCondensedConfiguration}
We show how to incorporate symmetries in $M$ to shrink down the information needed for the computation of the rank generating polynomial even more by introducing the \emph{condensed configuration}.

But first, we generalize the notion of the set of orbits of cyclic flats.
\begin{definition}
Let $P$ be a partition of $\EF{M}$. Then $P$ is called a \emph{condensation} of $\EF{M}$ if for all $B,C\in P$:
\begin{enumerate}
\item cardinality and rank are constant on $B$,
\item $\we_P(B,C):=|\{X \in B | X\subseteq Y\}|$ is independent of the choice of $Y\in C$. 
\end{enumerate}
For a condensation $P$ we set $B\leq_P C :\Leftrightarrow \we_P(B,C)>0$, for $B,C\in P$. Then $(P,\leq_P)$ becomes a lattice again.
\end{definition}
Notice that conditions 1. and 2. are automatically satisfied if $P$ is the set of $G$-orbits of $\EF{M}$ for $G\leq \Aut(M)$.
Another interesting example is to take $P$ as the set of equivalence classes of the relation $X\sim Y :\Leftrightarrow M|X\cong M|Y$. There can be many different condensations of $\EF{M}$; but since the  supremum of two configurations in the lattice of partitions of $\EF{M}$ (w.r.t. coarseness) is a configuration again, there is always a unique coarsest condensation.

By forgetting the concrete realization of a condensation $P$ as a partition of $\EF{M}$ we obtain something we want to call a \emph{condensed configuration} of $M$:
\begin{definition}
Let $P$ be a condensation of $\EF{M}$ and $\{R_B\}_{B\in P}$ a system of representatives.
Then the \emph{condensed configuration} (corresponding to P) of $M$ consists of
\begin{enumerate}
\item the matrix $(\we_P(B,C))_{B,C\in P}$ and
\item cardinality and rank of $R_B$, for all $B \in P$.
\end{enumerate}
\end{definition}
Especially if $M$ is highly symmetrical, a condensed configurations can give a really comprehensive overview of the arrangement of the cyclic flats in $M$ as the following examples illustrate:
\begin{example}
1.) If $P$ is the trivial partition $P=\{\{X\}|X\in \EF{M}\}$, then the condensed configuration corresponding to $P$ is the configuration of $M$, since $(\we_P(B,C))_{B,C\in P}$ is just the adjacency matrix of the lattice $(\EF{M},\subseteq)$.

2.) Consider the matroid $M=M_i$ ($i=1,2$) from Figure \ref{figConf} and denote the cyclic flats of rank 2 by $X_1$ and $X_2$. Then clearly $P=\{\{\emptyset\},\{X_1,\, X_2\},\linebreak[3] \{ \EM\}\}$ is a condensation of $\EF{M}$ and the corresponding condensed configuration can be represented by
\begin{align*}
 (\we_P(B,C))_{B,C\in P}&=\bordermatrix{%
        &\mbox{\small$(0, \,0)$}&\mbox{\small$(3, \,2)$}& \mbox{\small$(6, \,3)$} \cr
	  \mbox{\small$(0, \,0)$} &1 & 1 & 1\cr
	 \mbox{\small$(3, \,2)$} & 0 & 1 & 2\cr
	  \mbox{\small$(6, \,3)$}& 0 & 0 & 1\cr
},
\end{align*}
where each row and column is labeled by $(|X|,\, \rank{M}(X))$, $X\in B$.

3.) Let $M$ be the matroid induced by the extended binary \textsc{Golay} Code. Then $\Aut{M}$, the \textsc{Mathieu} group $M_{24}$, acts transitively on the cyclic flats of $M$ of given cardinality and rank and has six orbits on $\EF{M}$ altogether. The corresponding condensed configuration is
\begin{align*}
\bordermatrix{%
        &\mbox{\small$(0, \,0)$}&\mbox{\small$(8, \,7)$}& \mbox{\small$(12, \,10)$}& \mbox{\small$(12, \,11)$}& \mbox{\small$(16, \,11)$}& \mbox{\small$(24, \,12)$}\cr
	\mbox{\small$(0, \,0)$} &1       &1      &1      &1      &1      &1\cr
	\mbox{\small$(8, \,7)$} &0       &1      &3      &0      &30     &75\cr
	\mbox{\small$(12, \,10)$}&0       &0      &1      &0      &140    &35420\cr
           \mbox{\small$(12, \,11)$}&0       &0      &0      &1      &0      &2576\cr
           \mbox{\small$(16, \,11)$}&0       &0      &0      &0      &1      &759\cr
           \mbox{\small$(24, \,12)$}&0       &0      &0      &0      &0      &1\cr
}.
\end{align*}
From this we can for example read off that the cyclic flats of cardinality 12 and rank 11 are neither contained in nor contain any other non trivial cyclic flats in $M$. Using the axiom scheme to define a matroid by its cyclic flats introduced in \cite{Bonin}, we can safely remove this orbit of cyclic flats from $\EF{M}$ and obtain a new interesting matroid, which still has $M_{24}$ as automorphism group, but cannot be found ``in nature'' like $M$.
\end{example}
Again the definition of a condensed configuration is motivated by the following generalization of Theorem \ref{thmConfTutte}:
\begin{theorem}\label{thmCondConfTutte}
 The rank generating polynomial can be computed by a condensed configuration.
\end{theorem}
Recall that in the proof of Theorem \ref{thmConfTutte} we actually showed how to compute the cloud and flock polynomials in all minors $M|X/Y$ for cyclic flats $X\subset Y$. Generally, this is not possible in the case of condensed configurations, for the simple reason that from a condensed configuration we cannot derive exactly which minors appear in $M$.
But we will be able to compute \emph{average} cloud and flock polynomials.

So let $P$ be a condensation of $M$, $\{R_B\}_{B\in P}$ a system of representatives. To prove Theorem \ref{thmCondConfTutte}, we show how to compute the rank generating polynomial of $M$ only using the condensed configuration corresponding to $P$, i.e. the matrix $(\we_P(B,C))_{B,C\in P}$ and the cardinalities and ranks of the $R_B$, for $B \in P$.
\begin{definition}\label{definitionCFCondensation}
For $B,C\in P$ with $B\leq_P C$ we recursively define
\begin{align*}
  \cloudC{P}{B}{C}&:=\we_P(B,C)\bx{n}{r}-\dx(S(B,C)) \text{ and}\\
  \flockC{P}{B}{C}&:=\we_P(B,C)\by{n}{r}-\dy(S(B,C)), \text{ for}\\
  S(B,C)&:=\sum_{D} \cloudC{P}{D}{C}\flockC{P}{B}{D}
\end{align*}
 where $D$ ranges over $[B,C]_P - \{B,C\}$, $n:=|R_C|-|R_B|$ and $r:=\rank{M}(R_C)-\rank{M}(R_B)$.
\end{definition}
\begin{lemma}\label{lemmaCFCondensation}
In the notation of Definition \ref{definitionCFCondensation}:
\begin{align*}
  \cloudC{L}{B}{C}&=\sum_{\mathclap{\substack{X\in B\\X \subseteq R_C}}} \cloud{M|R_C}{X} \text{ and}\\
  \flockC{L}{B}{C}&=\sum_{\mathclap{\substack{X\in B\\X \subseteq R_C}}} \flock{M/X}{R_C-X}.
 \end{align*}
 Notice that this implies that the right hand sides are independent of the choice of $R_C \in C$, since the left hand sides are.
\end{lemma}
\begin{proof}
 We proof this by induction on $|[B,C]_P|$.
 By definition
$$
   \cloudC{P}{B}{C}=\we_P(B,C)\bx{n}{r}-\dx(S(B,C))
$$
 where 
$$
 S(B,C)=\sum_{D}\cloudC{P}{D}{C}\flockC{P}{B}{D}
$$
 and $D$ ranges over $[B,C]_P - \{B,C\}$.

 If $[B,C]_P = \{B,C\}$ then $S(B,C)=0$ and $\minor{M}{R_C}{X} \cong \Unif{r}{n}$, for all $X \in B$ with $X\subseteq R_C$, and the claim follows.

 Otherwise, since for all $D\in [B,C]_P - \{B,C\}$, $|[D,C]_P|$ and $|[B,D]_P|$ is less than $|[B,C]_P|$, we can apply induction and get
\begin{align*}
 S(B,C)&=\sum_{D}\;\; \sum_{\mathclap{\substack{Y\in D\\ Y\subseteq R_C}}}\; \cloud{M|R_C}{Y}\;\sum_{\mathclap{\substack{X\in B\\ X\subseteq R_D}}}\; \flock{M/X}{R_D-X}\\
       &=\sum_{D}\;\; \sum_{\mathclap{\substack{Y\in D\\ Y\subseteq R_C}}}\; \bigl[\cloud{M|R_C}{Y}\;\sum_{\mathclap{\substack{X\in B\\ X\subseteq R_D}}}\; \flock{M/X}{R_D-X}\bigl].\\
 \intertext{By induction the rightmost sum is independent of the choice of $R_D$, so}
       &=\sum_{D}\;\; \sum_{\mathclap{\substack{Y\in D\\ Y\subseteq R_C}}}\; \bigl[\cloud{M|R_C}{Y}\;\sum_{\mathclap{\substack{X\in B\\ X\subseteq Y}}}\; \flock{M/X}{Y-X}\bigl]\\
       &=\sum_{D}\;\; \sum_{\mathclap{\substack{Y\in D\\ Y\subseteq R_C}}}\;\;\;\sum_{\mathclap{\substack{X\in B\\ X\subseteq Y}}}\; \cloud{M|R_C}{Y} \flock{M/X}{Y-X}.\\
 \intertext{The three sums range over $\{(X,Y)|Y\in \EF{M}, X\in B\text{ with }X\subsetneq Y\subsetneq R_C \}$ and can hence be rearranged to}
       &=\sum_{\mathclap{\substack{X\in B\\X\subseteq R_C}}}\;\; \sum_{Y} \cloud{M|R_C}{Y}\; \flock{M/X}{Y-X},
\end{align*} 
where $Y\in [X,R_C]_{\EF{M}}-\{X,R_C\}$. Applying this to $\cloudC{P}{B}{C}$ yields
\begin{align*}
   \cloudC{P}{B}{C}&=\we_P(B,C)\bx{n}{r}\\&\;\;\;\;\;\;\;\;\;\;\;\;\;\;-\dx\bigl(\;\sum_{\mathclap{\substack{X\in B\\X\subseteq R_C}}}\;\sum_{Y} \cloud{M|R_C}{Y}\flock{M/X}{Y-X}\bigr)\\
		  &=\sum_{\mathclap{\substack{X\in B\\X\subseteq R_C}}}\;\;\bigl[\bx{n}{r}-\dx\bigl(\;\sum_{Y} \cloud{M|R_C}{Y}\flock{M/X}{Y-X}\bigr)\bigr].\\
\intertext{By Lemma \ref{lemmaCFESGS} and the definition of $n$ and $r$ this is}
		  &=\sum_{\mathclap{\substack{X\in B\\X\subseteq R_C}}}\;\cloud{M|R_C}{X}.
\end{align*} 
The statement for the flock polynomial follows analogously.
\end{proof}
Those \emph{average} cloud and flock polynomials suffice for the computation of the rank generating polynomial and analogously to the cloud/flock formula we obtain:
\begin{lemma}\label{lemmaAlg2}
  \begin{align*}
  \rgpM{M}=\sum_{B\in P} \cloudC{P}{B}{1_P} \flockC{P}{0_P}{B},
 \end{align*}
 where $1_P=\{\EM\}$ and $0_P=\{\varnothing\}$.
\end{lemma}
\begin{proof}
By Lemma \ref{lemmaCFCondensation}
\begin{align*}
 \sum_{B\in P} \cloudC{P}{B}{1_P}& \flockC{P}{0_P}{B}\\
     &=\sum_{B\in P}\;\sum_{\mathclap{\substack{X\in B\\X \subseteq \EM}}} \cloud{M|\EM}{X} \sum_{\mathclap{\substack{Y\in 0_P\\Y \subseteq R_B}}} \flock{M/Y}{R_B-Y}.\\
\intertext{Using the independence of choice of $R_B \in B$ we get}
     &=\sum_{B\in P}\;\sum_{\mathclap{X\in B}}\;\;\;\sum_{\mathclap{\substack{Y\in 0_P\\Y \subseteq X}}}\; \cloud{M}{X}  \flock{M/Y}{X-Y}.\\
\intertext{And since $0_P=\{\emptyset\}$}
     &=\sum_{B\in P}\;\sum_{\mathclap{X\in B}} \cloud{M}{X} \flock{M}{X}\\
     &=\rgpM{M}.\qedhere
\end{align*}
\end{proof}
The condensed configuration is constructed to contain all necessary information for the computation of the $\cloudC{P}{B}{C}$ and $\flockC{P}{B}{C}$. This proves Theorem \ref{thmCondConfTutte}.

%

\section{Condensed Configurations of Perfect Matroid Designs}\label{secPMD}
$M$ is called a \emph{perfect matroid design} if all flats in $M$ of given rank $i$ have the same cardinality $k_i$. In \cite{PMD} Mphako showed that the rank generating polynomial of a perfect matroid design is already determined by the numbers $k_i$. We show how to recover a condensed configuration of a perfect matroid design by the numbers $k_i$ as well. Combined with Theorem \ref{thmCondConfTutte} this yields a new prove for Mphako's result.

Let for now $M$ be a perfect matroid design of rank $r$, $k_i$ the cardinality of a flat of rank $i$ and $B_i$ the set of flats of rank $i$. Notice that $B_i$ consists of cyclic flats iff $k_i>k_{i-1}+1$. By dualizing the first formula in the proof of Theorem 3.6 in \cite{Cameron} we obtain that for all $i\leq j$ and $Y\in B_j$
\begin{align*}\label{eqPMDw}|\{X \in B_i |  X\subseteq Y\}|=\prod_{h=0}^{i-1} \frac{k_j-k_h}{k_i-k_h}.\end{align*}
Hence $P:=\{B_i|k_i>k_{i-1}+1\}$ is a condensation of $\EF{M}$. The condensed configuration of $M$ corresponding to $P$ only depends on the numbers $k_i$ since the $\we_P(B_i,B_j)$ and  the cardinalities and ranks of the cyclic flats are determined by the $k_i$. 
Summarizing this yields a new proof for:
\begin{theorem}[Mphako \cite{PMD}]
The rank generating (\textsc{Tutte}) polynomial of perfect matroid design only depends on the cardinalities and ranks of its flats.
\end{theorem}
Notice that we actually proved a stronger statement, since we can moreover compute the average cloud and flock polynomials $\cloudC{P}{B_i}{B_j}$ and $\flockC{P}{B_i}{B_j}$ now. This yields a new method to prove the nonexistence of certain perfect matroid designs. Firstly the coefficients of all average cloud and flock polynomials have to be positive integers. Secondly cardinality and rank of all flats which have to appear in the matroid can be determined by the exponents of the average cloud polynomials and may not differ from the $k_i$.
\bibliographystyle{plain} 
\bibliography{Bibliography} 
\end{document}